\documentclass[12pt]{article}
\usepackage{amsfonts}
\usepackage{mathrsfs}
\usepackage{amsmath}
\usepackage{amssymb}
\usepackage{graphicx}
\usepackage{epic}
\usepackage{color}

\usepackage[all,2cell]{xy}
\usepackage{xypic}
\usepackage{graphicx}
\usepackage[all]{xy}
\usepackage{enumerate}

\renewcommand{\paragraph}{\roman{paragraph}}

\usepackage{bm}
\usepackage{hyperref}
\usepackage{tikz}
\usetikzlibrary{arrows,shapes,positioning}
\usetikzlibrary{decorations.markings}
\tikzstyle arrowstyle=[scale=1]
\tikzstyle directed=[postaction={decorate,decoration={markings, mark=at position .65 with {\arrow[arrowstyle]{stealth}}}}]
\tikzstyle reverse directed=[postaction={decorate,decoration={markings, mark=at position .65 with {\arrowreversed[arrowstyle]{stealth};}}}]

\newtheorem{theorem}{Theorem}

\newtheorem{conjecture}[theorem]{Conjecture}

\newtheorem{problem}[theorem]{Problem}
\newtheorem{lemma}[theorem]{Lemma}

\newtheorem{proposition}[theorem]{Proposition}

\newenvironment{proof}{\noindent {\bf Proof.}}{\rule{3mm}{3mm}\par\medskip}

\begin{document}

\title{\bf The isomorphism problem of trees from the viewpoint of Terwilliger algebras \thanks{Supported by Natural Science Research Foundation of Anhui Provincial Department of Educationg(KJ2016A044), National Natural Science Foundation of China(11871071, 11771016, 11871073),
 Open Project of Anhui University (KF2019B03), Doctoral Start-up Fund of Anhui Jianzhu University(2018QD22).
  }}

\author{Shuang-Dong Li$^{1,2}$,
 Yi-Zheng Fan$^1$,
Tatsuro Ito $^1$\thanks{Tatasuro Ito: Corresponding author.},\\
Masoud Karimi$^3$,
Jing Xu $^1$\\
{\footnotesize \it $^1$ School of Mathematical Sciences, Anhui University, Hefei 230601, PR China}\\
{\footnotesize \it $^2$ Jianghuai College, Anhui University, Hefei 230031, PR China}\\
{\footnotesize \it $^3$ Department of Mathematics,Bojnourd Branch,Islamic Azad University,Bojnourd,Iran}
{\footnote{E-mail address: lisd@ahu.edu.cn (S.D.,Li); fanyz@ahu.edu.cn (Y.Z.,Fan); tito@staff.kanazawa-u.ac.jp (T.Ito);
karimimth@yahoo.com (M.Karimi); xujing@ahu.edu.cn (J.Xu).}}
}

\date{}
\maketitle

\begin{abstract}
Let $\Gamma^{(x_0)}$ be a finite rooted tree, for which $\Gamma$ is the underlying tree and $x_0$ the root.
 Let $T$ be the Terwilliger algebra of $\Gamma$ with respect to $x_0$.
 We study the structure of the principal $T$-module. As a result, it is shown that $T$ recognizes the isomorphism class of $\Gamma^{(x_0)}$.

\end{abstract}

\section{Introduction}
The Terwilliger algebras, which are oringinally called subconstituent algebras, are introduced in \cite{pt1} for association schemes, and their
 representations are deeply studied for ($P$ and $Q$)-polynomial association schemes in \cite{pt2, pt3}.
 In \cite{pt4}, the Terwilliger algebras are defined for graphs, and their representations are deeply studied for distance-regular graphs.
 This paper treats the Terwilliger algebras of trees and determines the structure of their principal modules. As a result, it is shown that the isomorphism classes of rooted trees are recognized by their Terwilliger algebras.

 Since there is no literature on Terwilliger algebras of graphs in general except for \cite{pt4} which is mainly aimed at distance-regular graphs with the $Q$-polynomial property, we begin with sorting out basic concepts about Terwilliger algebras, before going into details of our motivations (Conjecture 2, Problem 3) and main results (Theorem 4, Theorem 5).
 Notations and terminologies throughout the paper will be fixed in the course of doing so in this section.

\subsection{Terwilliger algebras in general as algebras}
 Let $V$ be a finite-dimensional vector space over the complex number field $\mathbb{C}$, equipped with a non-degenerate Hermitian form.
 We are given a decomposition of $V$ into the direct sum of mutually orthogonal subspaces $V_i^*,~0\leq i\leq D$,
 i.e., $V=\bigoplus \limits_{i=0}^{D}V^*_i$, $V^*_i\perp V^*_j$ ($i\neq j$).
 Let $E^*_i$ be the orthogonal projection from $V$ onto $V^*_i$: $I=E^*_0+E^*_1+\cdots +E^*_D$, $E^*_iE^*_j=\delta_{ij}E^*_i$, where $I$ is
 the identity map and $\delta_{ii}=1$, $\delta_{ij}=0~(i\neq j)$.
 We are also given a normal transformation $A$ of $V$, i.e., $V$ has an orthonormal basis
 consisting of eigenvectors of $A$.
 Let $T$ be the subalgebra of the endomorphism algebra $\mbox{End}(V)$ generated by $A, E^*_i,~0\leq i\leq D$:
 \begin{align}
 &T=\langle A,E_i^*~|~ 0\leq i\leq D\rangle \subseteq \mbox{End}(V).
 \end{align}
 The algebra $T$ is called a {\it Terwilliger algebra}, or simply a {\it $T$-algebra}.

 Here we remark that we use the notation of $E^*_i$, following \cite{pt1, pt4}, in which $E_i$ stands for the orthogonal projection onto the
 eigenspace $V_i$ of $A$ and $E^*_i$ does not mean the adjoint of $E_i$.
 In this paper, $V_i$, $E_i$ do not appear anywhere.

 Note that $T$ is a semi-simple algebra, since $T$ is generated by normal transformations of $V$.
 $V$ is called the {\it standard module}.
 Note also that $T$ acts on $V$ faithfully.
 So the standard module $V$ is a sum of irreducible $T$-submodules and every irreducible $T$-module appears in $V$ up to
 isomorphism.

 Let $W$ be a $T$-submodule of $V$.
 Then the restriction of $T$ to $W$ can be thought of a Terwilliger algebra for which $W$ is the standard module:
\begin{align}
&T|_W=\langle A|_W,{E_i^*}|_W~|~ 0\leq i\leq D\rangle \subseteq \mbox{End}(W),
\end{align}
 where some $E^*_i$ may vanish on $W$, i.e. ${E_i^*}|_W=0$ for some $i$'s.

 Let $T^{\prime}$ be another Terwilliger algebra generated by a normal transformation $A^{\prime}$, and orthogonal projections
 ${E^*_i}^{\prime}$, $0\leq i\leq D^{\prime}$, for which $V^{\prime}$ is the standard module.
 The Terwilliger algebra $T$ given in (1) is said to be isomorphic to $T^{\prime}$ if $D=D^{\prime}$ and
 the correspondence of $A$ to $A^{\prime}$, $E^*_i$ to ${E^*_i}^{\prime}$, $0\leq i\leq D$ gives
 an algebra isomorphism between $T$ and $T^{\prime}$.
 In this case, we denote $T\simeq T^{\prime}$ and such an algebra isomorphism is called a {\it $T$-algebra isomorphism}.
 In the definition of a $T$-algebra isomorphism,
 the orderings of the orthogonal projections $E^*_i, {E^*_i}^{\prime}$, $0\leq i\leq D$ matter, but usually as we shall see later for the Terwilliger algebra of a graph, there exists a natural ordering of them, which is implicitly assumed unless particularly stated.

The standard modules $V$, $V^{\prime}$ for the Terwilliger algebras $T$, $T^{\prime}$ are said to be {\it isomorphic} if $T\simeq T^{\prime}$ and
there exists an isometry $\varphi: V\rightarrow V^{\prime}$ between the Hermitian spaces $V$, $V^{\prime}$ such that $\varphi(av)=a^{\prime}\varphi(v)$ for all $a\in T, v\in V$,
where $a^{\prime}=f(a)$ with $f$ the $T$-algebra isomorphism from $T$ to $T^{\prime}$, i.e., the diagram
\begin{align}
&\xymatrix{
 V\ar[d]_a \ar[rr]^-\varphi && V^{\prime}\ar[d]^{a^{\prime}} \\
V\ar[rr]_\varphi && V^{\prime}}
 \end{align}
commutes for all $a\in T$. In this case, we denote $V\simeq V^{\prime}$.

\subsection{Terwilliger algebras for finite connected simple graphs}
In this subsection, we define the Terwilliger algebra for a finite connected simple graph and
pose some questions about it that have motivated our present study of the Terwilliger algebra of a tree.

Let $\Gamma$ be a finite connected simple graph, i.e., the graph $\Gamma$ is finite, undirected, connected and has no loops, no multiple edges.
Let $X$ be the vertex set of $\Gamma$ and $x_0$  a fixed vertex from $X$.
We call $x_0$ the {\it base vertex}.
The vertex set $X$ is partitioned according to the distance from the base vertex $x_0$: $X=\bigcup \limits_{i=0}^D X_i$,
\begin{align}
& X_i=\{x\in X~|~\partial(x_0,x)=i\}.
\end{align}
where $\partial(x_0,x)$ is the length of a shortest path joining $x_0$ and $x$, and $D=$ $\mbox{max}\{\partial(x_0,x)~|\\~x\in X\}$.

Set $V=\mathbb{C}X$, i.e., $V$ is the vector space over $\mathbb{C}$ formally spanned by $X$ as a basis.
Regarding $X$ as an orthonormal basis, $V$ is equipped with a non-degenerate Hermitian form.
The partition $X=\bigcup \limits_{i=0}^DX_i$ gives rise to an orthogonal decomposition
of $V$:~$V=\bigcup \limits_{i=0}^DV_i^*$, where $V_i^*=\mathbb{C}X_i$ is the subspace of $V$ spanned by $X_i$.

Let $A$ be the adjacency matrix of $\Gamma$, i.e., the $(x,y)$-entry of $A$ is $1$ if $(x,y)$ is an edge of $\Gamma$, $0$ otherwise for $x,y\in X$.
Then $A$ is a normal matrix, and it is regarded as a normal transformation of $V$.
Let $E^*_i$ be the orthogonal projection from $V$ onto $V^*_i$.
The subalgebra of $\mbox{End}(V)$ generated by $A$ and $E^*_i,~0\leq i\leq D$ is denoted by $T(x_0)$ and called the {\it Terwilliger algebra of $\Gamma$ } with respect to the base vertex $x_0$.
Note that the orthogonal projections have a natural ordering $E^*_0,E^*_1,\cdots,E^*_D$ according to the
distance of $X_i$ from the base vertex $x_0$. In what follows, we set $T=T(x_0)$.

The following proposition is from \cite{pt4}.
\begin{proposition}
Set $W_0=Tx_0$, the smallest $T$-submodule of $V$ containing the base vertex $x_0$.
Then $W_0$ is irreducible as a $T$-module.
\end{proposition}

This $T$-submodule $W_0=Tx_0$ is called the {\it principal $T$-module}, or the {\it primary $T$-module}.
The lecture notes \cite{pt4} are not officially published.
We quote a proof of the proposition from \cite{pt4}: Since $T$ is a semi-simple algebra, $W_0$ is a sum of irreducible $T$-submodules.
Apply $E^*_0$ to each irreducible $T$-submodule that appears in $W_0$. Then since $E^*_0W_0=V^*_0=\mathbb{C}x_0$,
there exists an irreducible $T$-submodule $W$ of $W_0$ such that $E^*_0W\neq 0$.
Then $W$ contains $x_0$ and so $W\supseteq Tx_0$ . This implies $W=W_0$.

The following conjecture is due to Jack Koolen and one of the motivations for our study of Terwilliger algebras of trees.
\begin{conjecture}
For almost all finite connected simple graphs $\Gamma$, the Terwilliger algebra $T=T(x_0)$ of $\Gamma$ coincides with the endomorphism algebra
$\rm{End}({\it V})$ of the standard module $V$, regardless of the base vertex $x_0: T=\rm{End}({\it V})$.
\end{conjecture}

In this paper, we precicely determine when $T=\mbox{End}(V)$ holds for a tree: $T=\mbox{End}(V)$ if and only if the tree does not have a symmetry with respect to the base vertex (see Theorem 4).

Another motivation of our study is the theme of the spectral graph theory: we shift from the adjacency algebra $\langle A\rangle$ of a graph
$\Gamma$ to the Terwilliger algebra $T=\langle A,E^*_i~|~0\leq i\leq D \rangle$, and ask how much the standard $T$-module $V$ determines the
graph structure of $\Gamma$. Particularly we ask:
\begin{problem}
Let $\Gamma$, ${\Gamma}^{\prime}$ be finite connected simple graphs and $T$, $T^{\prime}$  their Terwilliger algebras with $V$, $V^{\prime}$ the
standard modules, respectively.
When does $V\simeq V^{\prime}$ as $T$-modules imply $\Gamma \simeq {\Gamma}^{\prime}$ as graphs?
\end{problem}

If $V\simeq V^{\prime}$ implies $\Gamma \simeq {\Gamma}^{\prime}$, we say that the $T$-module $V$ {\it recognizes} the graph $\Gamma$.
When $V$, $V^{\prime}$ are the standard $T$-modules, we simply say $T$ recognizes $\Gamma$.
In this paper, we show that the Terwilliger algebra recognizes the tree (see Theorem 5).
Note that the adjacency algebra does not recognize the tree \cite{ajs}.

\subsection{Terwilliger algebras of trees}
In this subsection, we summarize our main results about the Terwilliger algebra of a tree.
In this paper, a tree means a finite tree.

Let $\Gamma$ be a tree and $X$  the vertex set of $\Gamma$.
Fix a base vertex $x_0\in X$ and form the Terwilliger algebra $T=T(x_0)$ with respect to $x_0$, in the way we explained in Section 1.2.

Let $V=\mathbb{C}X$ be the standard module. By ${\Gamma}^{(x_0)}$ we denote the rooted tree $\Gamma$ with $x_0$ the root.
Let $H$ be the automorphism group of ${\Gamma}^{(x_0)}$: $H=\mbox{Aut}({\Gamma}^{(x_0)})$.
So $H$ is the stabilizer of $x_0$ in the automorphism group of $\Gamma$: $G=\mbox{Aut}(\Gamma)$, $H=G_{x_0}$.
Note that the generators of $T$ commute with the action of $H$ on $V$. So we have
\begin{align}
&T\subseteq \mbox{Hom}_H(V,V),
\end{align}
where $\mbox{Hom}_H(V,V)$ is the centralizer algebra of $H$, i.e.,
the subalgebra of $\mbox{End}(V)$ consisting of all the linear transformations of $V$ that commute with the action of $H$ on $V$.

Suppose $T=\mbox{End}(V)$ holds in Conjecture 1.
Then we have $\mbox{Hom}{\it_H(V,V)}$=$\mbox{End}(V)$ by (5) and so $H=1$,
since $H\neq 1$ implies $\mbox{Hom}_H(V,V)\subsetneqq \mbox{End}(V)$.
We prove the converse holds in the case of trees, i.e., $T=\mbox{End}(V)$ if $H=1$.
Actually we show a stronger result as we see in the following theorem.
For a subset $Y$ of $X$, we denote by $\underline{Y}$ the sum of elements of $Y$: $\underline{Y}=\sum \limits_{y\in Y}y\in V$.
Let $W_0=Tx_0$ be the principal $T$-module. Then we have:
\begin{theorem}
$W_0$ is linearly spanned by $\underline{Y}$, $Y\in \mbox{Orb}(H,X)$, where $\mbox{Orb}(H,X)$ is the set of $H$-orbits on $X$:
$$W_0=\mbox{Span}\{\underline{Y}~|~Y\in \mbox{Orb}(H,X)\}.$$
\end{theorem}

Suppose $H=1$, then $W_0=V$ by this theorem.
So by Proposition 1, the standard module $V$ is irreducible as a $T$-module.
If $T\subsetneqq \mbox{End}(V)$, $V$ can not be an irreducible $T$-module by Burnside's Theorem.
Therefore $H=1$ implies $T=\mbox{End}(V)$.

As for Problem 3, we also show a stronger result, i.e., we show that the principal $T$-module $W_0$ recognizes the tree.
Note that if the principal $T$-module $W_0$ recognizes the tree, so does the standard $T$-module $V$.

Suppose we are given a finite connected simple graph ${\Gamma}^{\prime}$ besides the tree $\Gamma$.
Let $T$, $T^{\prime}$ be the Terwilliger algebras of $\Gamma$, ${\Gamma}^{\prime}$  and $W_0$, ${W_0}^{\prime}$ the principal $T$-modules of  $\Gamma$, ${\Gamma}^{\prime}$, respectively.
Recall that $T|_{W_0}$, $T^{\prime}|_{{W_0}^{\prime}}$ become Terwilliger algebras with standard modules $W_0$, ${W_0}^{\prime}$, respectively,
and that there is a natural ordering for the generators of $T|_{W_0}$, $T^{\prime}|_{{W_0}^{\prime}}$, respectively.
\begin{theorem}
If $W_0\simeq {W_0}^{\prime}$ as $T$-modules, then ${\Gamma}^{\prime}$ is a tree.
Moreover we have ${\Gamma}^{(x_0)}\simeq {{\Gamma}^{\prime}}^{({x_0}^{\prime})}$ as rooted trees,
where $x_0$, ${x_0}^{\prime}$ are the base vertices of $\Gamma$, ${\Gamma}^{\prime}$, respectively.
\end{theorem}

Theorem 4 and Theorem 5 will be proved in Section 3 and Section 4.

\section{The key lemma}

We keep the notations in Section 1.3.

Let $\Gamma$ be a finite tree and $X$  the  vertex set of $\Gamma$.
Fix a base vertex $x_0\in X$.
Then we have the partition $X=\bigcup \limits_{i=0}^{D}{X_i}$  given by (4). For $x\in X_i$, set
\begin{align}
&{\Gamma}^{(x)}=\{y\in X~|~\partial(x_0,y)=\partial(x_0,x)+\partial(x,y)\},
\end{align}
where $\partial$ stands for the distance function of  $\Gamma$ as in (4).
Allowing abuse of the notation, we regard ${\Gamma}^{(x)}$ as the rooted tree with $x$ the root that is induced on the subset ${\Gamma}^{(x)}$ of $X$.

We introduce an equivalence relation $\sim$ on $X_i$ by defining $x\sim x^{\prime}$ for $x$, $x^{\prime}\in X_i$
if and only if ${\Gamma}^{(x)}$ and ${\Gamma}^{(x^{\prime})}$ are isomorphic as rooted trees.
Let $X_i(\alpha)$, $\alpha \in {\Lambda}_i$ be the equivalence classes on $X_i$:
\begin{align}
&X_i=\bigcup \limits_{\alpha \in {\Lambda}_i}X_i(\alpha).
\end{align}
Accordingly we have a decomposition of the subspace $V_i^*=\mathbb{C}X_i$ into the direct sum of mutually orthogonal subspaces $V^*_i(\alpha)=\mathbb{C}X_i(\alpha), ~\alpha \in {\Lambda}_i$:
\begin{align}
&V^*_i=\bigoplus \limits_{\alpha \in {\Lambda}_i}V^*_i(\alpha).
\end{align}
Let $E^*_i(\alpha)$ be the orthogonal projection from $V=\mathbb{C}X$ onto $V^*_i(\alpha)$. Then we have
\begin{align}
&E^*_i=\sum \limits_{\alpha \in {\Lambda}_i}E^*_i(\alpha)
\end{align}
with $E^*_i(\alpha)E^*_i(\beta)={\delta}_{\alpha \beta}E^*_i(\alpha)$,
where ${\delta}_{\alpha \beta}=1$ if $\alpha = \beta,~ 0$ otherwise.
Note that we have $I=\sum \limits_{i=0}^{D} \sum \limits_{\alpha \in {\Lambda}_i}E^*_i(\alpha)$
with $E^*_i(\alpha)E^*_j(\beta)={\delta}_{i j}{\delta}_{\alpha \beta}E^*_i(\alpha),$ where $I$ is the identity map.

The following lemma is the key to our proofs of Theorem 4 and Theorem 5.
Recall that $A$ is the adjacency matrix of the tree $\Gamma$ and
the Terwilliger algebra $T=T(x_0)$ of $\Gamma$ is the subalgebra of $\mbox{End}(V)$ generated by $A,~E^*_i$, $0\leq i\leq D$.

\begin{lemma}
For $0\leq i\leq D-1$, the subalgebra of $\rm{End}({\it V})$ generated by $E^*_i(\alpha)$, $\alpha \in {\Lambda}_i$
coincides with that generated by $E^*_iAE^*_{i+1}(\beta)AE^*_i$, $\beta \in {\Lambda}_{i+1}$:
\begin{align}
&\langle E^*_i(\alpha)~|~ \alpha \in {\Lambda}_i \rangle =\langle E^*_iAE^*_{i+1}(\beta)AE^*_i ~|~ \beta \in {\Lambda}_{i+1} \rangle.
\end{align}
In particular, the Terwilliger algebra $T$ contains every $E^*_i(\alpha)$:
\begin{align}
&E^*_i(\alpha) \in T,~0\leq i\leq D,~\alpha \in {\Lambda}_i.
\end{align}
\end{lemma}
\begin{proof}
For $x \in X_i$ and $\beta \in {\Lambda}_{i+1},$ let $n_{\beta}(x)$ denote the number of vertices in $X_{i+1}(\beta)$ that are adjacent to $x$:
\begin{align}
&n_{\beta}(x)=\sharp \{y\in X_{i+1}(\beta)~|~x~and~y~are~adjacent~in~\Gamma\}.
\end{align}
Observe that each $x$ in $X_i$ is an eigenvector of $ E^*_iAE^*_{i+1}(\beta)AE^*_i$ and it belongs to the eigenvalue $n_{\beta}(x)$:
\begin{align}
&E^*_iAE^*_{i+1}(\beta)AE^*_i x= n_{\beta}(x)x.
\end{align}
Therefore $V^*_i= \mathbb{C}X_i$ ~is decomposed into the direct sum of maximal common eigensp-\\aces of $E^*_iAE^*_{i+1}(\beta)AE^*_i$, $\beta \in {\Lambda}_{i+1}$, each of which has a basis consisting of some elements from $X_i$.

For $x$, $x^{\prime} \in X_i$, it holds that ${\Gamma}^{(x)}$  and ${\Gamma}^{(x^{\prime})}$ are isomorphic as rooted trees
if and only if $n_{\beta}(x)=n_{\beta}(x^{\prime})$ for all $\beta \in {\Lambda}_{i+1}$.
Therefore  $x$, $x^{\prime} \in X_i$ belong to the same $X_i(\alpha)$ for some $\alpha \in {\Lambda}_i$
if and only if they belong to the same maximal common eigenspaces of $E^*_iAE^*_{i+1}(\beta)AE^*_i$, $\beta \in {\Lambda}_{i+1}$,
i.e., $V^*_i(\alpha)=\mathbb{C}X_i(\alpha)$, $\alpha \in {\Lambda}_i$ are the maximal common eigenspaces of $E^*_iAE^*_{i+1}(\beta)AE^*_i$, $\beta \in {\Lambda}_{i+1}$.
Since $E^*_i(\alpha)$ is the projection from $V$ onto $V^*_i(\alpha)$, we have (10).

Observe $|{\Lambda}_D|=1$ and so $E^*_D(\alpha)=E^*_D$, i.e., (11) holds for $i=D$. By induction on $i$,
we may assume the right hand side of (10) is contained in $T$.
Then the left hand side is contained in $T$, i.e., (11) holds for all $i$.
\end{proof}

\section{The structure of the principal $T$-module of a tree: Proof of Theorem 4}
In this section, we prove Theorem 4, which gives the structure of the principal $T$-module of a tree.

Recall that $H$ is the automorphism group of the rooted tree ${\Gamma}^{(x_0)}$.
Let $\mbox{Irr}(H, X)$ denote the set of irreducible characters of $H$
that appear in the permutation representation of $H$ acting on the vertex set $X$ of ${\Gamma}^{(x_0)}$.
So the standard module $V=\mathbb{C}X$ affords the permutation character of $H$.
For $\chi \in \mbox{Irr}(H, X),$ let $V_{\chi}$ denote the sum of irreducible $H$-submodules of $V$ that
afford $\chi$.
Then we have the homogeneous component decomposition of the $H$-module $V$, i.e.,
$V$ is the direct sum of $V_{\chi}, \chi \in \mbox{Irr}(H, X)$:
\begin{align}
&V=\bigoplus \limits_{\chi \in \rm{Irr}({\it H, X})}V_{\chi}.
\end{align}
Let ${\chi}_0=1_H,$ the trivial character of $H$. Then we have
\begin{align}
&V_{{\chi}_0}=\mbox{Span} \{\underline{Y}~|~ Y\in \mbox{Orb}(H,X)\},
\end{align}
where $\mbox{Orb}(H,X)$ is the set of $H$-orbits on $X$ and $\underline{Y}=\sum \limits_{y\in Y}y \in V$ for $Y\in \mbox{Orb}(H,X).$

Recall that $\mbox{Hom}_H(V,V)$ is the centralizer algebra of the $H$-module $V.$ Set $S=$\\ $\mbox{Hom}_H(V,V).$
For the centralizer algebra $S$ of a finite group $H,$
it is well-known in general that $S$ is a semi-simple algebra and that the homogeneous component
decomposition of the $S$-module $V$ is also given by (14), i.e.,
any irreducible $S$-submodule of $V$ is contained in some $V_{\chi}$,
and irreducible $S$-submodules of $V$ are isomorphic if and only if they belong to the same $V_{\chi}$.
It is also well-known that $V_{{\chi}_0}$ is irreducible as an $S$-module.

The root $x_0$ of the tree $\Gamma$ is fixed by $H$ and so it is contained in $V_{{\chi}_0}$ by (15).
Since $V_{{\chi}_0}$ is irreducible as an $S$-module, we have $V_{{\chi}_0}=Sx_0.$
On the other hand, we have $T\subseteq S$ by (5).
This implies that the principal $T$-module $W_0=Tx_0$ is contained in $V_{{\chi}_0}=Sx_0$: $W_0\subseteq V_{{\chi}_0}.$
We want to show $ W_0\supseteq V_{{\chi}_0}$.
By (15), it is enough to show $W_0\ni \underline{Y}$ for all $Y\in \mbox{Orb}(H,X)$.
We may assume $Y\subseteq X_i(\alpha)$ for some $i,~\alpha$:
observe that if $x,~x^{\prime}$ belong to the same $H$-orbit, then $\partial(x_0,x)=\partial(x_0,x^{\prime})$, and ${\Gamma}^{(x)}$, ${\Gamma}^{(x^{\prime})}$ are isomorphic as rooted trees.

Since $\Gamma$ is a tree, we have the adjacency mapping from the set $X_i$ given in (4) to the set $X_{i-1}$:
$$\psi: X_i \rightarrow X_{i-1},~x \mapsto \psi(x),$$
where $\psi(x)$ is the unique element of $X_{i-1}$ that is adjacent to $x \in X_i$.
It is obvious that if $x$, $x^{\prime} \in X_i$ belong to the same $H$-orbit,
then so do $\psi(x)$, $\psi (x^{\prime}) \in X_{i-1}$.
Conversely suppose $\psi (x)$, $\psi (x^{\prime}) \in X_{i-1}$ belong to the same $H$-orbit.
Then $x$, $x^{\prime} \in X_i$ belong to the same $H$-orbit if and only if there exists $\alpha \in {\Lambda}_i$ such that
$x$, $x^{\prime} \in X_i(\alpha),$ i.e., ${\Gamma}^{(x)}$  and ${\Gamma}^{(x^{\prime})}$ are isomorphic as rooted trees.

Let $Y$ be a $H$-orbit in $X_i(\alpha)$. Set $Z=\psi (Y)$, where $\psi (Y)=\{\psi (y)~|~y\in Y\}$.
Then by the argument in the previous paragraph, $Z$ is a $H$-orbit in $X_{i-1}$,
and $E^*_i(\alpha){\psi}^{-1}(Z)$ is a $H$-orbit in $X_i(\alpha)$,
where ${\psi}^{-1}(Z)=\{x \in X_i~|~\psi (x) \in Z\}$ and
$E^*_i(\alpha){\psi}^{-1}(Z)=\{E^*_i(\alpha)x~|~ x \in {\psi}^{-1} (Z)\}= X_i(\alpha)\bigcap {\psi}^{-1} (Z)$.
We have $Y=E^*_i(\alpha){\psi}^{-1}(Z)$,
since $Y$ is contained in $E^*_i(\alpha){\psi}^{-1}(Z)$
and both $Y$ and $E^*_i(\alpha){\psi}^{-1}(Z)$  are a $H$-orbit.
Note that for $z\in Z$, $E^*_iAz$ is the sum of elements in the set
${\psi}^{-1}(z)=\{x\in X_i~|~\psi(x)=z\}$: $E^*_iAz=\underline{{\psi}^{-1}(z)}=\sum \limits_{x\in {\psi}^{-1}(z)}x$.
Thus we have
\begin{align}
&\underline{Y}=E^*_i(\alpha)A\underline{Z}
\end{align}
for a $H$-orbit $Y$ in $X_i(\alpha)$, where $Z=\psi(Y)$. Note that $Z=\psi(Y)$ is a $H$-orbit in $X_{i-1}$.
For $i=0$,~~$\underline{Y}=x_0$ is contained in $W_0=Tx_0$.
By induction on $i$,~$\underline{Y}$ is contained in $W_0$ for all
$i$ by (16), since $E^*_i(\alpha) \in T$ by Lemma 6. This completes the proof of Theorem 4.

\section{The principal $T$-module recognizes the tree: Proof of Theorem 5}
In this section, we prove Theorem 5, which states that the principal $T$-module recognizes the tree.

We firstly prove that ${\Gamma}^{\prime}$ is a tree if the principal $T^{\prime}$-module ${W_0}^{\prime}$ of ${\Gamma}^{\prime}$ is
isomorphic to the principal $T$-module $W_0$ of the tree $\Gamma$.
Since $\Gamma$ is a tree, we have
\begin{align}
&E^*_0A^iE^*_iA^iE^*_0x_0={\| E^*_iA^iE^*_0x_0\|}^2x_0,\\
&E^*_0A^iE^*_iAE^*_iA^iE^*_0x_0=0
\end{align}
for $0\leq i\leq D$, where $\|w\|$ is the norm of $w\in W_0.$
Since $W_0\simeq {W_0}^{\prime},$ we have $D=D^{\prime}$ and (17), (18) hold for ${x_0}^{\prime},~{E^*_0}^{\prime},~A^{\prime},~{E^*_i}^{\prime}$
in place of $x_0,~E^*_0,~A,~E^*_i$. This means that ${\Gamma}^{\prime}$ does not have any cycles,
i.e., ${\Gamma}^{\prime}$ is a tree.

Next we prove that ${\Gamma}^{(x_0)}$ and ${{\Gamma}^{\prime}}^{({x_0}^{\prime})}$ are isomorphic as rooted trees if $W_0\simeq {W_0}^{\prime}$
as $T$-modules.
Note that Lemma 6 holds for the principal $T$-module $W_0$ (and for the principal $T^{\prime}$-module ${W_0}^{\prime}$),
i.e., (10) and (11) hold for $E^*_i(\alpha)|_{W_0},~ E^*_{i+1}(\beta)|_{W_0},~ E^*_i|_{W_0},$ \\$A|_{W_0}$
in place of $E^*_i(\alpha),~ E^*_{i+1}(\beta), ~E^*_i,~ A$.
Note that $E^*_i(\alpha)|_{W_0}\neq 0$, $0\leq i\leq D$, ${\alpha}\in {\Lambda}_i$, since $X_i(\alpha)$ is invariant under the action of
$H$ and so $\underline{Y}\in W_0$ for $Y\in \mbox{Orb}(H,X_i(\alpha))$ by Theorem 4.
Since $T|_{W_0}\simeq T^{\prime}|_{{W_0}^{\prime}}$ as $T$-algebras, it follows from Lemma 6 by induction on $D-i$ that we can identify
${\Lambda}_i$ and ${{\Lambda}_i}^{\prime}$, $0\leq i\leq D=D^{\prime}$ and $E^*_i(\alpha)|_{W_0}$ corresponds to  ${E^*_i}^{\prime}(\alpha)|_{{W_0}^{\prime}}, ~0\leq i\leq D, ~\alpha \in{\Lambda}_i$ under the $T$-algebra isomorphism between $T|_{W_0}$ and $T^{\prime}|_{{W_0}^{\prime}}$.

For $x\in X_1(\alpha), ~\alpha \in {\Lambda}_1$, we consider the Terwilliger algebra of the rooted tree ${\Gamma}^{(x)}$.
Set $\hat{V}=\bigoplus \limits_{y\in X_1}\mathbb{C}{\Gamma}^{(y)}=\bigoplus \limits_{i=1}^DV^*_i$ and let $\hat{T}$ denote the subalgebra of
$\mbox{End}(\hat{V})$ generated by $E^*_i, 1\leq i\leq D$ and $\hat{A},$ where
\begin{align}
&\hat{A}=(\sum \limits_{j=1}^DE^*_j)A(\sum \limits_{i=1}^DE^*_i).
\end{align}
Then $\hat{T}$ is a Terwilliger algebra with $\hat{V}$ the standard module, in the sense of Section 1.1.
The Terwilliger algebra of ${\Gamma}^{(x)}$
for $x\in X_1(\alpha)$ is $\hat{T}|_{\mathbb{C}{\Gamma}^{(x)}}$,
the restriction of $\hat{T}$ to the standard module $\mathbb{C}{\Gamma}^{(x)}$ of
${\Gamma}^{(x)},$ and $\hat{T}x$ is the principal module of ${\Gamma}^{(x)}$.
For $x$, $x^{\prime}\in X_1(\alpha)$, we have $\hat{T}x\simeq \hat{T}x^{\prime}$
as $\hat{T}$-modules since ${\Gamma}^{(x)}$ and ${\Gamma}^{(x^{\prime})}$ are isomorphic as rooted trees.
Set $\underline{X_1(\alpha)}=\sum \limits_{x^{\prime}\in X_1(\alpha)}x^{\prime}$. Then for any $x\in X_1(\alpha),$ we have
\begin{align}
&\hat{T}\underline{X_1(\alpha)}\simeq \hat{T}x
\end{align}
as $\hat{T}$-modules with the natural correspondence of $a \underline{X_1(\alpha)}$ to $ax$ for $a\in \hat{T}$.
Note that $\underline{X_1(\alpha)}=E^*_1(\alpha)Ax_0$ and so $\hat{T}\underline{X_1(\alpha)}$ is contained in $W_0=Tx_0$ by Lemma 6.
Therefore by (20), the principal module $\hat{T}x$ of ${\Gamma}^{(x)}$ for $x\in X_1(\alpha)$ is determined up to isomorphism by the principal
module $W_0=Tx_0$ of ${\Gamma}^{(x_0)}$, as $\hat{T}$ is a sualgebra of $T$.

Recall it follows from $T|_{W_0}\simeq T^{\prime}|_{{W_0}^{\prime}}$ that ${\Lambda}_i$ and ${{\Lambda}_i}^{\prime}$ are identified and $E^*_i(\alpha)|_{W_0}$ corresponds to~${E^*_i}^{\prime}(\alpha)|_{{W_0}^{\prime}}$ $(\alpha \in {\Lambda}_i)$ under the $T$-algebra isomorphism between $T|_{W_0}$ and $T^{\prime}|_{{W_0}^{\prime}}$.
Hence by (20), the principal module of ${\Gamma}^{(x)}$ is isomorphic to that of ${{\Gamma}^{\prime}}^{(x^{\prime})}$
for $x\in X_1(\alpha), ~x^{\prime}\in {X_1}^{\prime}(\alpha),~\alpha \in {\Lambda}_1$, i.e., $\hat{T}x\simeq {\hat{T}}^{\prime}x^{\prime}$,
because $\hat{T}\underline{X_1(\alpha)}\simeq {\hat{T}}^{\prime}\underline{X_1^{\prime}(\alpha)}$.
By induction on $D$, ${\Gamma}^{(x)}$ is isomorphic to ${{\Gamma}^{\prime}}^{(x^{\prime})}$ as rooted trees
for $x\in X_1(\alpha),~ x^{\prime}\in {X_1}^{\prime}(\alpha),~ \alpha \in {\Lambda}_1$, i.e., ${\Gamma}^{(x)}\simeq {{\Gamma}^{\prime}}^{(x^{\prime})}$,  since the $D$ of ${\Gamma}^{(x)}$ is smaller than that of ${\Gamma}^{(x_0)}$.

The isomorphism class of the rooted tree ${\Gamma}^{(x_0)}$ is determined by $\{{\Gamma}^{(x)} ~|~ x\in X_1\}$ regarded as a multi-set of isomorphism
classes of rooted trees. Let $\{{\Gamma}^{(\alpha)}~ |~ \alpha\in {\Lambda}_1\}$ be a complete set of representatives for the isomorphism classes of the rooted trees ${\Gamma}^{(x)}$, $x\in X_1$.
Then $\{{\Gamma}^{(x)} ~|~ x\in X_1\}=\{|X_1(\alpha)|{\Gamma}^{(\alpha)} ~|~ \alpha \in {\Lambda}_1\}$ as multi-sets
and the isomorphism class of the rooted tree ${\Gamma}^{(x_0)}$ is determined by $\{|X_1(\alpha)|~|~ \alpha \in {\Lambda}_1\}$
and $\{{\Gamma}^{(\alpha)} ~|~ \alpha\in {\Lambda}_1\}$.

We have already shown that $W_0\simeq {W_0}^{\prime}$ implies ${\Lambda}_1= {{\Lambda}_1}^{\prime}$ and ${\Gamma}^{(\alpha)}={{\Gamma}^{\prime}}^{(\alpha)},~ \alpha \in {\Lambda}_1$.
Since $\underline{X_1(\alpha)}=E^*_1(\alpha)Ax_0$ and $|X_1(\alpha)|$ coincides with the norm of $\underline{X_1(\alpha)}$,
$W_0\simeq {W_0}^{\prime}$ implies $|X_1(\alpha)|=|{X_1}^{\prime}(\alpha)|$, $\alpha \in {\Lambda}_1$.
This completes the proof of Theorem 5.

%%%%%%%%%%%%%%%%%%%%%%%%%%%%%%%%%%%%%%%%%%%%%%%%%%%%%%%%%%

\end{document}